\documentclass{article}

\usepackage{amsmath,amssymb,amsthm}
\usepackage{tikz-cd}
\usepackage{enumerate}
\usepackage{mathtools}
\usepackage{appendix}
\usepackage{subfigure}
\usepackage{url}
\usepackage{hyperref}
\usepackage{xurl}

\newtheorem{lem}{Lemma}[section]

\newtheorem{rem}[lem]{Remark}

\newtheorem{thm}[lem]{Theorem}

\newtheorem{conj}{Conjecture}

\title{Rigidity, flexibility, and a sign error}
\author{Yank\i\ Lekili\ and\ Nick Sheridan}

\begin{document}

\maketitle

\begin{abstract}
    The documents \cite{openai1,openai2,openai3} claimed to prove cases of the Hodge conjecture via homological mirror symmetry, before being withdrawn by OpenAI. 
    We discuss the source of the error (but do not analyse the rest of the argument).
\end{abstract}

\section{Summary}

The attempted proofs of cases of the Hodge conjecture via homological mirror symmetry in the documents \cite{openai1,openai2,openai3}, withdrawn by OpenAI, all broke down in the same way, at the point where they attempted to construct an unobstructed Lagrangian immersion in a symplectic manifold. They constructed a graded Lagrangian immersion using the Gromov--Lees $h$-principle, and claimed that it could be modified so that all self-intersection points were of odd degree, by \cite[Theorem 1.1]{EEMS}. 
This modification was crucial for the ensuing argument that the immersion was unobstructed, because only self-intersection points of even degree, in fact only those of degree 2, contribute to the `obstruction space'.

In fact, inspection of the proof of \cite[Theorem 1.1]{EEMS} (or the variations proved in \cite{openai1,openai2,openai3}) shows that it always produces an immersion with a distinguished self-intersection point $p$ of degree $2$; in particular, it can never eliminate all even self-intersection points. 
Applying \cite[Theorem 1.1]{EEMS} correctly only allows to modify the Lagrangian immersion so that all of the self-intersection points \emph{except for $p$} are of odd degree. The presence of the degree-$2$ self-intersection point $p$ causes the proof of unobstructedness to break down. 

Not only does the proof of unobstructedness break down, the Lagrangian immersions produced by \cite[Theorem 1.1]{EEMS} are typically obstructed. More precisely, there is a leading term in the obstruction, arising from a rigid holomorphic disc ending on $p$; we explain why this causes the Lagrangian to be obstructed in a certain regime. Even outside this regime, there is no reason to expect that this term in the obstruction would be cancelled.

The error arose because \cite{openai1,openai2,openai3} computed signs of self-intersection points with the standard convention, whereas \cite{EEMS} implicitly use the opposite convention (this was apparently first pointed out and reported to OpenAI by Nathan Humphrey \cite{blogpost,github}). 
The authors of \cite{EEMS} do not say explicitly that their results are stated using the opposite of the standard convention; on the other hand, they use this convention consistently, the convention can be deduced from formulae in the paper and from reading the proof, and the results are correct if their choice of convention is understood. 

The obstructedness of the Lagrangian immersions produced by \cite[Theorem 1.1]{EEMS} is consistent with a heuristic principle in symplectic topology, namely the dichotomy between \emph{flexibility} (by which is roughly meant, problems which can be solved purely using techniques from differential topology, such as Gromov's $h$-principle \cite{Gromov:hprin}) and \emph{rigidity} (by which is roughly meant, problems for which such techniques do not suffice); cf. \cite{Yasha}. 
The clarification of the border between flexibility and rigidity is a key question in the field, but unobstructed Lagrangians are generally rigid. 
Typically one can get from something rigid to something flexible, but not the other way round. 
This caused several researchers to be sceptical of the rather general claimed construction of unobstructed Lagrangian immersions via $h$-principles in \cite{openai1,openai2,openai3}, before they were withdrawn. 
On the other hand, such intuitions can of course be mistaken -- \cite[Theorem 1.1]{EEMS} itself being an example of a result which contradicted expectations on the rigidity of self-intersection numbers of Lagrangian immersions.

We find that the release of the documents \cite{openai1,openai2,openai3} has led to the following positive outputs:
\begin{enumerate}
    \item The approach to cases of the Hodge conjecture via homological mirror symmetry which they pursue is interesting and worth pursuing; it had been proposed in \cite{Katzarkov:HMSHodge} and pursued by various mathematicians in the past without success. The construction of unobstructed Lagrangians is the key difficulty. If this is overcome, the documents \cite{openai1,openai2,openai3} may well provide a useful guide to other technical aspects of the problem. 
    \item The fact that the results of \cite{EEMS} are stated using the opposite of the standard sign convention has been clarified.
    \item The question of whether the Lagrangian immersions constructed in \cite{EEMS} can be unobstructed has been clarified.
\end{enumerate}

\paragraph{Acknowledgements:} We thank Stefan Nemirovski for drawing the reference \cite{Katzarkov:HMSHodge} to our attention. N.S.'s work is supported by the Leverhulme Trust through the Leverhulme Prize, and a Simons Investigator award (award number 929034).

\section{The problem and the strategy of proof}\label{sec:hodge}

We are concerned with (see \cite{Deligne2006Hodge}):

\begin{conj}[Hodge]\label{conj:hodge}
For every smooth complex projective variety $X$, and every class $\alpha \in H^{2p}(X;\mathbb{Q}) \cap H^{p,p}(X)$, $\alpha$ can be expressed as a rational linear combination of classes $\mathrm{cl}(Z)$ of algebraic cycles $Z$.
\end{conj}

The documents \cite{openai1,openai2,openai3} aim to prove cases of Conjecture \ref{conj:hodge}: i.e., to prove the statement for certain varieties $X$ and classes $\alpha$. 
Modulo important technical details (which, however, are not the subject of this note and which we do not evaluate), the approach is as follows. 

First, they show that each variety $X$ that they consider satisfies Kontsevich's homological mirror symmetry conjecture \cite{kontsevich:hms}.\footnote{The key works to establish the relevant cases of Kontsevich's conjecture are due to Abouzaid \cite{Abouzaid:semiflat} and Seidel \cite{Seidel:quartic}.} 
I.e., there exists a specific symplectic $2n$-manifold $Y$ (called the `mirror' of $X$) and a quasi-equivalence  
\begin{equation}\label{eq:hms}
    \mathrm{D^bCoh}(X) \simeq \mathrm{Fuk}(Y).
\end{equation}

In order to represent a cohomology class $\alpha \in H^{p,p}(X)$ by an algebraic cycle, it then becomes sufficient to represent a corresponding homology class $\beta \in H_n(Y)$ by an unobstructed Lagrangian immersion: i.e., to find an unobstructed Lagrangian immersion $f:L \to Y$ satisfying $f_*[L] = \beta$. 

This is a valid approach to representing Hodge classes by algebraic cycles.\footnote{For an approach to constructing novel algebraic cycles by showing a mirror Lagrangian is unobstructed, without reference to the Hodge conjecture, see \cite{Hicks} -- note that this is a different kind of Lagrangian to those considered in \cite{openai1,openai2,openai3}.}   
The problem is that constructing unobstructed Lagrangians is also a hard problem.  
It was in the construction of this unobstructed Lagrangian immersion $f:L \to Y$ that the proofs in \cite{openai1,openai2,openai3} broke down. 

\section{Where the proof of unobstructedness breaks down}\label{sec:unob}

First, we explain what is meant by unobstructedness. 
Given a graded spin Lagrangian immersion $f:L \to Y$, with transverse self-intersection points 
$$\mathrm{SInt}(f)\coloneq \{p \in L: \exists q \in L, q \neq p, f(p) = f(q)\},$$
there is an associated curved $A_\infty$ algebra 
\begin{align*}
    A^*_L &\coloneqq \bigoplus_{j \in \mathbb{Z}} A^j_L,\qquad\text{where}\\
    A^j_L& = C^j(L;\Lambda) \oplus \bigoplus_{p \in \mathrm{SInt}(f), |p| = j} \Lambda \cdot p
\end{align*}
linear over a certain valued field $\Lambda$, the Novikov field. 
The grading of $L$ determines a degree $|p| \in \mathbb{Z}$ for each $p \in \mathrm{SInt}(f)$, which is the degree in which it contributes to $A^*_L$, in accordance with \cite{Seidel:graded}.
In other words, there are filtration-respecting $\Lambda$-linear operations
$$\mathfrak{m}_k: A_L^{\otimes k} \to A_L$$
of degree $2-k$, for $k \ge 0$, satisfying the $A_\infty$ relations, defined by `counting' holomorphic discs in $Y$ with boundary on $L$; see \cite{AkahoJoyce,Fukaya:immersed} for the construction. 

A \emph{bounding cochain} for $L$ is an element $b \in A^1_L$ of positive valuation, which is a solution to the Maurer--Cartan equation
\begin{equation}\label{eq:MC}
\mathfrak{m}_0^b \coloneq \sum_{k \ge 0} \mathfrak{m}_k(b,\ldots,b) = 0.
\end{equation}
Note that the bounding cochain $b$ lives in degree $1$, while the output $\mathfrak{m}_0^b$ of the Maurer--Cartan equation (the `obstruction') lives in degree $2-k+k=2$.

If $L$ admits a bounding cochain, it is called \emph{unobstructed}. 
By definition, an object of the Fukaya category $\mathrm{Fuk}(Y)$ is a Lagrangian immersion $f:L \to Y$ equipped with a bounding cochain $b \in A^1_L$. 
That is why we needed to construct an \emph{unobstructed} Lagrangian in the previous section: only unobstructed Lagrangians define objects of the category showing up in the equivalence \eqref{eq:hms}, which is how information gets transferred from $Y$ to $X$.

The documents \cite{openai1,openai2,openai3} propose to use the following variation of a result of Fukaya--Oh--Ohta--Ono to construct an unobstructed Lagrangian immersion:\footnote{The documents \cite{openai1,openai2,openai3} each produce their own variation on Fukaya--Oh--Ohta--Ono's proof.}

\begin{thm}[Variation on Theorem 3.8.11 of \cite{fooo:1}]
\label{thm:unobs}
    Suppose that $Y$ is a $2n$-dimensional symplectic manifold, and $f: L \to Y$ a graded Lagrangian immersion such that the map $f_*: H_{n-2}(L;\mathbb{Q}) \to H_{n-2}(Y;\mathbb{Q})$ is injective, $[\mathrm{GW}_{0,1}(Y)(L)] = 0 \in H_{n-2}(Y; \mathbb{Q})$ (where $\mathrm{GW}_{0,1}(Y)(L)$ denotes the two-pointed, genus-zero Gromov--Witten invariant of $Y$ evaluated at the image of $L$), and that no self-intersection point $p \in \mathrm{SInt}(f)$ has degree $2$. 
    Then $L$ is unobstructed.
\end{thm}

The proof is by a deformation-obstruction argument. At each step one gets an obstruction class in $A^2_L$, which we must show vanishes. The vanishing of $[\mathrm{GW}_{0,1}(Y)(L)]$ and the injectivity of $f_*$ in degree $n-2$ is sufficient to prove that the component in $C^2(L;\Lambda)$ vanishes; and the non-existence of self-intersection points of degree $2$ implies that they do not contribute to the obstruction space, i.e. that $A^2_L = C^2(L;\Lambda)$, so this implies that the obstruction class vanishes. The vanishing of the Gromov--Witten invariant $[\mathrm{GW}_{0,1}(Y)(L)]$ is automatic for the spaces $Y$ considered in \cite{openai1,openai2,openai3}. 

For any Lagrangian immersion $f:L \to Y$, where $Y$ is $4k$-dimensional, we define the signed self-intersection number $J(f)$ by
$$2J(f) \coloneqq f_*[L] \cdot f_*[L] - (-1)^k \chi(L).$$
where $Y$ is given its symplectic orientation. It is equal to $0$ if $f$ is an embedding. 
Each self-intersection point $p \in \mathrm{SInt}(f)$ contributes $(-1)^{k + |p|}$ to $2J(f)$. 

In order to produce a Lagrangian submanifold $L \subset Y$ satisfying the conditions of Theorem \ref{thm:unobs}, and with the desired homology class $f_*[L] = \beta$, the idea is first to apply the Gromov--Lees $h$-principle \cite{Gromov:hprin,Lees:hprin} to construct an oriented Lagrangian immersion $f:L \to Y$ with $(-1)^k J(f)<0$ and $f_*[L] = \beta$.  
One then wishes to apply the following result of Ekholm--Eliashberg--Murphy--Smith:

\begin{thm}[Theorem 1.1 of \cite{EEMS}]\label{thm:eems}
Suppose that $Y$ is a simply-connected symplectic $4k$-manifold with $k>1$, and $f:L \to Y$ a Lagrangian immersion of a closed oriented $2k$-manifold. Then there exists another Lagrangian immersion $f':L \to Y$, Hamiltonian regular homotopic to $f$, with transverse self-intersections, and
$$\frac{|\mathrm{SInt}(f')|}{2} = \begin{cases} 
|I(f)| & \text{if $(-1)^{k}I(f) < 0$};\\
|I(f)|+2 & \text{otherwise,}
\end{cases}$$
where $I(f) = -J(f)$.\footnote{$J(f)$ is the Whitney self-intersection number for the orientation $\omega^{2k}$ of $Y$; the number $I(f)$ of \cite{EEMS} is its negative, as one sees from the formula $I(f_0)$$ = (-1)^{n/2}\chi(L)/2$ in \cite[§4.2]{EEMS}, and as is forced by \cite[Corollary 1.3(3)]{EEMS}, since $\chi(L) < 0$ is the flexible case.}
\end{thm}

The reason for the difference of $+2$ between the two cases will be explained at the start of the next section. 
(As the spaces $Y$ considered in \cite{openai1,openai2,openai3} are not simply connected, they provide an alternative argument to prove an analogous result in the situation where they need it, but that is not relevant for the rest of this discussion.)

The problem that arose in \cite{openai1,openai2,openai3} was that they assumed that $I(f) = J(f)$, rather than $-J(f)$. 
We now explain why this was fatal to the argument. 

We will call a point $p \in \mathrm{SInt}(f)$ \emph{odd} if $|p|$ is odd, and \emph{even} if $|p|$ is even. 
Note that Theorem \ref{thm:eems} says that, as $(-1)^k I(f) = -(-1)^k J(f) > 0$, we can modify the immersion $f$ to produce an immersion $f'$ with $|I(f)|+1$ odd self-intersection points, and one even self-intersection point, call it $p$. 
The odd self-intersection points cannot contribute to $A^2_L$; but the even self-intersection point  $p$ will contribute two even generators to $A^*_L$, and we cannot rule out the possibility that one of them might have degree $2$, hence violate the hypothesis of Theorem \ref{thm:unobs}. 
Indeed, we will see in the next section that the construction always produces a self-intersection point of degree $2$, which will cause the Lagrangian $L$ to be obstructed.

Had the signs been the other way round, i.e., had the statement of Theorem \ref{thm:eems} been with $J(f)$ in place of $I(f)$, it would have said that the modified immersion $f'$ had $|I(f)|$ odd self-intersection points, none of which could have had degree $2$; so the hypothesis of Theorem \ref{thm:unobs} would have been satisfied and guaranteed that the Lagrangian immersion was unobstructed.

\begin{rem}
    The document \cite{openai2} used a variation on the argument, where they resolved all the self-intersection points of $f$ to obtain a Lagrangian embedding, which, however, was only $\mathbb{Z}/4$-graded, and applied a corresponding modification of Theorem \ref{thm:unobs}. 
     This argument fails because surgery at a self-intersection point $p$ will only produce a $\mathbb{Z}/2N$-graded Lagrangian if $|p| \equiv 1$ mod $2N$; in particular, one needs all self-intersection points to be odd.
     Thus, the even self-intersection point provided by the correct statement of Theorem \ref{thm:eems} will destroy the grading of the Lagrangian.
\end{rem}

\section{Why the Lagrangian immersion is typically obstructed}\label{sec:ob}

The Lagrangian immersion $f':L \to Y$ of Theorem \ref{thm:eems} (or the analogous constructions from \cite{openai1,openai2,openai3}) is produced from $f:L \to Y$ as follows. 
First, one finds a Darboux ball $B = B_{2r} \subset Y$, symplectomorphic to the standard ball of radius $2r$ in the symplectic vector space $\mathbb{R}^{2n}$, such that $f(L)$ corresponds to $\mathbb{R}^n \cap B_{2r}$. 
One then modifies $f$ inside $B_{2r}$ so that it has a standard form inside $B_r$: namely, it is an exact filling of the stabilization $\Lambda$ of the Legendrian unknot, which is conical over the complement of a smaller ball $B_\sigma \subset B_r$, and has a single transverse self-intersection point $p$ of degree $2$ inside $B_\sigma$. 

The point of this modification is that it brings us into the world of flexibility, as $\Lambda$ is loose (being a stabilization), see \cite{Murphy:thesis}. 
This looseness of $\Lambda$ allows us to modify the part of $f$ lying outside of $B_r$ to remove as many intersection points as are allowed by topology, arranging that all self-intersection points outside $B_r$ have the same sign; but we are left with the degree-$2$ self-intersection point $p$ inside $B_r$, in either case.
This is the reason for the difference of $2$ between the two cases of Theorem \ref{thm:eems}: if $(-1)^kI(f)<0$ then we may arrange that all self-intersection points including $p$ are even, but if $(-1)^kI(f)\ge 0$ then we may arrange that all self-intersection points are odd \emph{except for $p$}. 
The degree-$2$ self-intersection $p$ is the cost of flexibility, and as we will now see, it also gives rise to the obstructedness of the resulting Lagrangian immersion.

\begin{lem}\label{lem:smalldisc}
    The signed count of holomorphic discs mapping to $B_\sigma$, with boundary on $L$, and a single corner at $p$, is $\pm 1$.
\end{lem}
\begin{proof}[Proof (sketch)]
    Observe that we can arrange that there are three Reeb chords from $\Lambda$ to itself: $c$, $c_1$, and $c_2$ \cite[Example 4.7]{EES}. 
    Furthermore, the differential of $c_2$ in the Chekanov--Eliashberg differential graded algebra is $\partial(c_2) = \pm 1$ (see \cite[Proof of Proposition 4.8]{EES}). 
This implies that the signed count of the moduli space of holomorphic discs mapping to the symplectization $\partial B_\sigma \times \mathbb{R}$, having boundary on $\Lambda \times \mathbb{R}$ and asymptotic to the Reeb chord $c_2$ along a boundary puncture, modulo translation, is $\pm 1$. 

We now consider the moduli space of holomorphic discs mapping into the completion of $B_\sigma$ along its boundary, having boundary on the completion of $L$, asymptotic to $c_2$ along a boundary puncture.  
This moduli space will be one-dimensional and admits a compactification where one end (corresponding to the disc going `to infinity') is identified with the moduli space above, and the other corresponds to broken configurations consisting of two discs: one with one boundary puncture asymptotic to $c_2$ and one corner at $p$, and the other with one corner at $p$ (cf. \cite{SFTcompactness}).

As the signed count of points at one end of the compactification is $\pm 1$, so is the signed count at the other end, consisting of the broken configurations. 
The latter is the product of signed counts of the two moduli spaces of discs making up the broken configuration; hence the signed count of both (and in particular the count of discs with a corner at $p$) is $\pm 1$. 
\end{proof}

We claim that for $\sigma$ sufficiently small, $f:L \to Y$ is obstructed. 
By exactness of $L \cap B_r$, every holomorphic disc in $B_r$ with boundary on $L$ and a single corner at $p$ has area equal to the action $a(p)$ of the double point, and by Lemma \ref{lem:smalldisc} the signed count of these is $\pm 1$. 
As the disc lives in the ball of radius $\sigma$, we have $a(p) = O(\sigma^2)$, whereas by monotonicity every holomorphic disc with a corner at $p$ that leaves $B_r$ has area bounded below by a constant depending only on $B_r$.  Let $b\in A^1_L$ have positive valuation. A disc with output $p$ that leaves $B_r$ has area greater than $a(p)$ for $\sigma$ sufficiently small. A disc with output $p$ contained in $B$ has area exactly $a(p)$, and any insertion of $b$ along the boundary raises the valuation of the contribution above $a(p)$. Hence, the coefficient of $p$ in $\mathfrak{m}_0^b$ is $\pm T^{a(p)}$ plus terms of higher valuation, for every $b \in A^1_L$ of positive valuation. So $\mathfrak{m}_0^b \neq 0$, and $L$ is obstructed.

For the Lagrangian immersion to have a chance of being unobstructed, it would thus be necessary to increase $\sigma$ enough that the low-energy term $\pm T^{a(p)}$ can cancel with a contribution from some other holomorphic disc. 
However, we note that there is an upper bound $r$ on how much we can increase $\sigma$, and furthermore, this bound will typically be small: as $f:L \to Y$ is constructed via an $h$-principle, it will typically have a lot of `wrinkles' in it, meaning we can't choose the ball $B_{2r}$ very large (the largest possible radius of this ball is called the \emph{Gromov width} of $L$). 
Even if we could take it large, there would need to be other holomorphic discs ending on $p$ and with area exactly $a(p)$ whose contributions would need to cancel those from the discs produced by Lemma \ref{lem:smalldisc}, and there is no reason why such holomorphic discs should exist.

\bibliographystyle{alphaurl}
\bibliography{references.bib}

\vspace{\fill} 
\noindent
\begin{tabular}{@{}l}
  \textit{Department of Mathematics}, \textit{Imperial College, London}, \\ \texttt{y.lekili@imperial.ac.uk} \\
  \textit{School of Mathematics, University of Edinburgh, and the Maxwell Institute}, \\ \texttt{nick.sheridan@ed.ac.uk}
\end{tabular}

\end{document}